\documentclass[10pt,leqno,letterpaper]{article}
\usepackage[utf8]{inputenc}
\usepackage{amsmath}
\usepackage{amsfonts}
\usepackage{amssymb}
\usepackage{graphicx}
\usepackage{mathrsfs}
\usepackage{upref,amsthm,amsxtra,exscale}
\usepackage{cite}
\usepackage[colorlinks=true,urlcolor=blue,
citecolor=red,linkcolor=blue,linktocpage,pdfpagelabels,
bookmarksnumbered,bookmarksopen]{hyperref}
\usepackage{fullpage}

\usepackage{subcaption}
\usepackage{caption}

\newtheorem{theorem}{Theorem}[section]

\newtheorem{remark}[theorem]{Remark}

\newtheorem{lemma}[theorem]{Lemma}
\newtheorem{proposition}[theorem]{Proposition}
\newtheorem{definition}[theorem]{Definition}

\numberwithin{equation}{section}

\usepackage{enumitem}

\def\N{\mathbb{N}}
\def\R{\mathbb{R}}
\def\S{\mathbb{S}}
\def\cP{\mathcal{P}}
\def\cM{\mathcal{M}}

\def\cC{\mathcal{C}}

\def\r{\mathbb{R}}
\def\rn{\mathbb{R}^N}
\def\sn{\mathbb{S}^N}

\def\n{\mathbb{N}}

\def\eps{\varepsilon}
\def\o{\Omega}

\def\irn{\int_{\r^N}}

\def\tilde{\widetilde}

\def\cC{\mathcal{C}}

\def\cH{\mathcal{H}}

\def\cM{\mathcal{M}}
\def\cN{\mathcal{N}}
\def\cP{\mathcal{P}}
\def\cJ{\mathcal{J}}
\def\cT{\mathcal{T}}
\def\cU{\mathcal{U}}

\def\sL{\mathscr{L}}
\def\sP{\mathscr{P}}
\def\bf{\mathbf}

\author{Mónica Clapp\footnote{M. Clapp was  supported by CONACYT grant A1-S-10457 (Mexico).}, Juan Carlos Fernández\footnote{J.C. Fernández was supported by a CONACYT postdoctoral fellowship (Mexico).}, and Alberto Saldaña\footnote{
A. Saldaña was supported by UNAM-DGAPA-PAPIIT grant IA101721 (Mexico).}}
\title{Critical polyharmonic systems and optimal partitions}
\date{\today}

\AtEndDocument{\bigskip{\footnotesize% 
		\textbf{Mónica Clapp}\par
		\textsc{Instituto de Matemáticas, Universidad Nacional Autónoma de México} \par  
		\text{Circuito Exterior, Ciudad Universitaria, 04510 Coyoacán, Ciudad de México, Mexico}\par
		\textit{E-mail address:} \texttt{monica.clapp@im.unam.mx} \par
		\addvspace{\medskipamount}
		\textbf{Juan Carlos Fernández}\par
		\textsc{Instituto de Matemáticas, Universidad Nacional Autónoma de México} \par  
		\text{Circuito Exterior, Ciudad Universitaria, 04510 Coyoacán, Ciudad de México, Mexico}\par
		\textit{E-mail address:} \texttt{jcfmor@im.unam.mx} 
		\par
		\addvspace{\medskipamount}
		\textbf{Alberto Saldaña}\par
		\textsc{Instituto de Matemáticas, Universidad Nacional Autónoma de México} \par  
		\text{Circuito Exterior, Ciudad Universitaria, 04510 Coyoacán, Ciudad de México, Mexico}\par
		\textit{E-mail address:} \texttt{alberto.saldana@im.unam.mx} 
}}

\begin{document}

\maketitle

\begin{abstract}
We establish the existence of solutions to a weakly-coupled competitive system of polyharmonic equations in $\mathbb{R}^N$ which are invariant under a group of conformal diffeomorphisms, and study the behavior of least energy solutions as the coupling parameters tend to $-\infty$.  We show that the supports of the limiting profiles of their components are pairwise disjoint smooth domains and solve a nonlinear optimal partition problem of $\R^N$. We give a detailed description of the shape of these domains.

\medskip


\medskip

\end{abstract}

\section{Introduction}

Consider the weakly-coupled competitive polyharmonic system of $\ell$ equations in $\mathbb{R}^N$,
\begin{equation} \label{s:intro}
	\begin{cases}
		(-\Delta)^m u_i = \mu_i|u_i|^{2_m^*-2}u_i + \sum\limits_{\substack{j=1 \\ j\neq i}}^\ell\lambda_{ij}\beta_{ij}|u_j|^{\alpha_{ij}}|u_i|^{\beta_{ij}-2}u_i,\quad i=1,\ldots,\ell, \\
		u_i\in D^{m,2}(\rn),\qquad i=1,\ldots,\ell,
	\end{cases}
\end{equation}
where $m,N\in\N$, $N>2m$,  $2_m^*:=\frac{2N}{N-2m}$ is the critical Sobolev exponent, $\mu_i>0$, $\lambda_{ij}=\lambda_{ji}<0$, $\alpha_{ij},\beta_{ij}>1$ satisfying $\alpha_{ij}=\beta_{ji}$ and $\alpha_{ij}+\beta_{ij}=2_m^*$, and $D^{m,2}(\rn)$ is the completion of $\cC_c^\infty(\rn)$ with respect to the norm $\|\cdot\|$ induced by the scalar product
\begin{equation} \label{eq:scalar_product}
\langle u,v\rangle :=
\begin{cases}
\irn\Delta^\frac{m}{2}u\cdot\Delta^\frac{m}{2}v, &\text{for \ }m\text{ even},\\
\irn\nabla\Delta^\frac{m-1}{2}u\cdot\nabla\Delta^\frac{m-1}{2}v, &\text{for \ }m\text{ odd}.
\end{cases}
\end{equation}

In this paper we establish the existence of solutions to \eqref{s:intro} which are invariant under some groups of conformal diffeomorphisms, and study the behavior of least energy solutions as $\lambda_{ij}\to -\infty.$  We show that the supports of the limiting profiles of their components are pairwise disjoint smooth domains and solve a nonlinear optimal partition problem in $\R^N$.

To state our results we introduce some notation. Fix $n_1,n_2\in\n$ with $n_1,n_2\geq 2$ and $n_1+n_2=N+1$, and set $\Gamma:=O(n_1)\times O(n_2)$. Each $\gamma\in\Gamma$ is an isometry of the unit sphere $\sn$, and gives rise to a conformal diffeomorphism $\tilde\gamma:\rn\to\rn$ given by $\tilde\gamma x:=(\sigma\circ\gamma^{-1}\circ\sigma^{-1})(x)$, where $\sigma:\sn\to\rn\cup\{\infty\}$ is the stereographic projection.  A subset $\o$ of $\rn$ will be called \emph{$\Gamma$-invariant} if $\tilde\gamma x\in\o$ for all $x\in\o$, and a function $u:\o\to\r$ will be said to be \emph{$\Gamma$-invariant} if $$|\det\tilde\gamma'(x)|^\frac{1}{2^*_m}u(\tilde\gamma x)=u(x)\qquad \text{for all \ }\gamma\in\Gamma, \ x\in\o.$$

If $\Omega$ is a $\Gamma$-invariant open subset of $\rn$ we write $D_0^{m,2}(\Omega)$ for the closure of $\cC_c^\infty(\Omega)$ in $D^{m,2}(\rn)$ and we use $D_0^{m,2}(\Omega)^\Gamma$ and $D^{m,2}(\rn)^\Gamma$ to denote the subspaces of $\Gamma$-invariant functions in $D_0^{m,2}(\Omega)$ and $D^{m,2}(\rn)$ respectively. Consider the Dirichlet problem
\begin{equation} \label{eq:dirichlet}
\begin{cases}
(-\Delta)^m u = |u|^{2_m^*-2}u,\\
u\in D_0^{m,2}(\Omega)^\Gamma.
\end{cases}
\end{equation}
This problem has a least energy nontrivial solution (see Section \ref{sec:segregation}), whose energy will be denoted by $c^\Gamma_\Omega$, i.e.,
$$c^\Gamma_\Omega:=\inf\left\{\frac{m}{N}\|u\|^2:u\neq 0\text{ and } u\text{ solves }\eqref{eq:dirichlet}\right\}.$$
We consider partitions of $\rn$ by $\Gamma$-invariant open subsets. More precisely, for $\ell\geq 2$, let
\begin{align*}
\cP_\ell^\Gamma:=\{\{\Omega_1,\ldots,\Omega_\ell\}:&\,\Omega_i\neq\emptyset \text{ is a }\Gamma\text{-invariant open subset of }\rn\text{ and }\Omega_i\cap \Omega_j=\emptyset\text{ if }i\neq j \}.
\end{align*}
A \emph{$(\Gamma,\ell)$-optimal partition} for $\mathbb R^N$ is a partition $\{\Omega_1,\ldots,\Omega_\ell\}\in\cP_\ell^\Gamma$ such that
\begin{align}\label{op}
\sum_{i=1}^\ell c^\Gamma_{\Omega_i}=\inf_{\{\Theta_1,\ldots,\Theta_\ell\}\in\cP_\ell^\Gamma}\sum_{i=1}^\ell c^\Gamma_{\Theta_i}. 
\end{align}

We study a symmetric version of \eqref{s:intro}, namely 
\begin{equation} \label{eq:system}
\begin{cases}
(-\Delta)^m u_i = \mu_i|u_i|^{2_m^*-2}u_i + \sum\limits_{\substack{j=1 \\ j\neq i}}^\ell\lambda_{ij}\beta_{ij}|u_j|^{\alpha_{ij}}|u_i|^{\beta_{ij}-2}u_i,\quad i=1,\ldots,\ell, \\
u_i\in D^{m,2}(\rn)^\Gamma,\qquad i=1,\ldots,\ell,
\end{cases}
\end{equation}
where $\lambda_{ij}$, $\mu_i$, $\alpha_{ij}$ and $\beta_{ij}$ are as before. 

Our first result asserts the existence of infinitely many fully nontrivial solutions of \eqref{eq:system}. A solution $(u_1,\ldots,u_\ell)$ to the system \eqref{eq:system} is called \emph{fully nontrivial} if each component $u_i$ is nontrivial. We refer to Definition \ref{def:least energy} for the notion of a least energy fully nontrivial solution.

\begin{theorem} \label{thm:existence}
The system \eqref{eq:system} has a least energy fully nontrivial solution and a sequence of fully nontrivial solutions which is unbounded in $[D^{m,2}(\rn)]^\ell$. 
\end{theorem}

Our next result describes the segregation behavior of least energy fully nontrivial solutions as $\lambda_{ij}\to -\infty$, showing that the supports of the limiting profiles of their components solve the optimal partition problem \eqref{op}. We write $\mathbb{S}^{d-1}$ and $\mathbb{B}^d$ for the unit sphere and the open unit ball in $\r^d$. The symbol $\cong$ stands for ``is $\Gamma$-diffeomorphic to".

\begin{theorem} \label{thm:main}
For $i=1,\ldots,\ell$, fix $\mu_i=1$  and for each $i\neq j$, $k\in\n$, let $\lambda_{ij,k}<0$ be such that $\lambda_{ij,k}=\lambda_{ji,k}$ and $\lambda_{ij,k}\to -\infty$ as $k\to\infty$. Let $(u_{k,1},\ldots,u_{k,\ell})$ be a least energy fully nontrivial solution to the system \eqref{eq:system} with $\lambda_{ij}=\lambda_{ij,k}$. Then, after passing to a subsequence, we have that
\begin{itemize}
\item[$(a)$]$u_{k,i}\to u_{\infty,i}$ strongly in $D^{m,2}(\rn)$, $u_{\infty,i}\in\cC^{m-1}(\rn)$, and $u_{\infty,i}\neq 0$.  Let 
\begin{align*}
 \Omega_i:=\operatorname{int}\overline{\{x\in\rn:u_{\infty,i}(x)\neq 0\}}\qquad \text{ for \ }i=1,\ldots,\ell.
\end{align*}
Then $u_{\infty,i}\in \cC^{2m,\alpha}(\overline\Omega_i)$ is a least energy solution of \eqref{eq:dirichlet} in $\Omega_i$ for each $i=1,\ldots,\ell$.
\item[$(b)$]$\{\Omega_1,\ldots,\Omega_\ell\}\in\cP_\ell^\Gamma$ is a $(\Gamma,\ell)$-optimal partition for $\mathbb R^N$.
\item[$(c)$]$\Omega_1,\ldots,\Omega_\ell$ are smooth and connected, $\overline{\Omega_1\cup\cdots\cup \Omega_\ell}=\rn$ and, after relabeling, we have that
\begin{itemize}
\item[$(c_1)$] $\Omega_1\cong\mathbb{S}^{n_1-1}\times \mathbb{B}^{n_2}$, \ $\Omega_i\cong\mathbb{S}^{n_1-1}\times\mathbb{S}^{n_2-1}\times(0,1)$ if  $i=2,\ldots,\ell-1$, and \ $\Omega_\ell\cong\mathbb{B}^{n_1}\times \mathbb{R}^{n_2-1}$,
\item[$(c_2)$] $\overline{\Omega}_i\cap \overline{\Omega}_{i+1}\cong\mathbb{S}^{n_1-1}\times\mathbb{S}^{n_2-1}$ and\quad $\overline{\Omega}_i\cap \overline{\Omega}_j=\emptyset$\, if\, $|j-i|\geq 2$.
\end{itemize}
\end{itemize}
\end{theorem}

Combining Theorems \ref{thm:existence} and \ref{thm:main} we obtain the following result.

\begin{theorem} \label{teo:partition}
For every group $\Gamma:=O(n_1)\times O(n_2)$ with $n_1,n_2\geq 2$ and $n_1+n_2=N+1$, there exists a $(\Gamma,\ell)$-optimal partition for $\mathbb R^N$ having the properties stated in $(c)$ above.
\end{theorem}

Theorem \ref{thm:existence} extends the multiplicity result in \cite{BaScWe} for a single polyharmonic equation to systems, and generalizes the results for systems involving the Laplacian in \cite{cp} and \cite[Theorem 1.2]{cs} to the higher-order case. In all of these results the symmetries play a crucial role in compensating the lack of compactness inherent to critical problems. This fact was first used by W.Y. Ding in \cite{d}. For bounded domains with Dirichlet boundary conditions, critical polyharmonic systems with linear and subcritical coupling terms have been studied in \cite{BaGu,Mo}, whereas critical couplings were considered in \cite{GoRa}. See also \cite{Ta} for some results on weakly-coupled fourth-order Schrödinger equations.

For the proof of Theorem \ref{thm:existence} we follow the variational approach introduced in \cite{cs} which carries over immediately to higher-order operators and may be used to obtain existence and multiplicity results for other polyharmonic systems as well.

The connection between optimal partitions and competitive systems for the Laplacian was first noticed in \cite{ctv} and has been further developed in various papers considering different types of nonlinearities, couplings, and general smooth domains. Optimal partitions and shape optimization problems in general are difficult to study in the higher-order regime. The available results for the Laplacian involve the use of advanced tools such as Almgren-type monotonicity formulae, boundary Harnack principles, and Liouville theorems. The extention of all this machinery to the higher-order case seems out of reach. For general statements and a review of previously known results for the Laplacian we refer to \cite{sttz}. 

In Theorem \ref{thm:main} we make strong use of the symmetries to obtain and fully describe the shape of the optimal partition. This result extends the main theorem in \cite{css}. As far as we know, it is the first result to exhibit and fully characterize an optimal partition for a higher-order elliptic operator. 

The conformal invariance of the system \eqref{eq:system} allows translating it into the polyharmonic system on the standard sphere,
\begin{equation} \label{eq:system_sn}
\begin{cases}
\sP^m_g v_i = \mu_i |v_i|^{2^*_m-2}v_i + \sum_{i\neq j}\lambda_{ij}\beta_{ij}\vert v_j\vert^{\alpha_{ij}}\vert v_i\vert^{\beta_{ij}}v_i,\\
v_i\in H^m_g(\S^N),\\
v_i\text{ is }\Gamma\text{-invariant},\qquad i,j=1,\ldots,\ell,
\end{cases}
\end{equation}
with the same $\mu_i$, $\lambda_{ij},\alpha_{ij},\beta_{ij}$, where $\sP^m_g$ is a conformally invariant operator generalizing the conformal Laplacian for $m=1$ and the Paneitz operator for $m=2$, see \eqref{Eq:JGM-Operator}. More precisely, $\bar v=(v_1,\ldots,v_\ell)$ is a solution of \eqref{eq:system_sn} iff $\bar u=(\iota(v_1),\ldots,\iota(v_\ell))$ solves \eqref{eq:system} where $\iota$ is defined in terms of the stereographic projection, see Proposition \ref{prop:equivalent_spaces1}. Theorems \ref{thm:existence}, \ref{thm:main} and \ref{teo:partition} translate into similar statements for the system \eqref{eq:system_sn}, which is interesting in itself.

The orbit space of the action of $\Gamma$ on $\sn$ is one-dimensional, see \eqref{eq:orbit_map}. This allows to translate the systems \eqref{eq:system_sn} and \eqref{eq:system} into a system of ODE's. However, the operator has a rather complicated expression and it is degenerate, in the sense that it involves sign-changing weights that can vanish at different points. 

For $m=1$, the ODE approach was exploited in \cite{fp} to derive the existence of sign-changing solutions to the Yamabe equation on the sphere having precisely $\ell$ nodal domains for any $\ell\geq 2$, using a double-shooting method. This result does not extend easily to $m\geq 2$, because the corresponding ODE has a rather complicated expression, see Remark \ref{b:rmk} below.  On the other hand, in \cite{css}, sign-changing solutions to the Yamabe problem with a prescribed number of nodal domains were constructed using an alternating sum of limiting profiles of \emph{positive} least energy solutions to \eqref{eq:system} with $m=1$. This method also fails when considering $m\geq 2$, since it is not known if the least energy solutions of \eqref{eq:system} are signed or sign-changing.  Therefore, the existence of sign-changing solutions to the problem
\begin{equation} \label{eq:problem}
\sP^m_g v = |v|^{2^*_m-2}v,\qquad v\in H^m_g(\S^N),
\end{equation}
with precisely $\ell$ nodal domains for any $\ell\geq 2$ remains an open question.

Problem \eqref{eq:problem} arises naturally in conformal geometry when seeking for prescribed higher-order conformal invariants, called $Q$-curvatures, generalizing the scalar curvature \cite{GGS,FeGr,Ro}. For $m=1$ it is the Yamabe problem  and it is the Paneitz problem for $m=2$ \cite{DjHeLe,LePa}.

This paper is organized as follows. In Section \ref{Section:regularity} we use the symmetries to restore compactness and to derive regularity properties of the $\Gamma$-invariant functions. Next, in Section \ref{sec:system}, we describe the variational setting for the polyharmonic system and prove Theorem \ref{thm:existence}. Finally, in Section \ref{sec:segregation} we study the behavior of the least energy solutions to the system as $\lambda_{ij}\to-\infty$ and prove Theorem \ref{thm:main}. To simplify our presentation, two technical results are added in an Appendix.

\section{Compactness and regularity by symmetry}\label{Section:regularity}

Let $(\S^N,g)$ denote the unit sphere with its round metric. For $m\in\N$ and $N>2m$, the Sobolev space $H^m_g(\S^N)$ is the completion of $\cC^\infty(\S^N)$ with respect to the norm defined by the interior product
\begin{equation}\label{Eq:Standard Norm}
\langle u, v\rangle_{H_g^m(\sn)}:=
\begin{cases}
\int_{\S^N}( uv + \Delta_g^{m/2}u\cdot\Delta_g^{m/2}v) \;dV_g, & m \text{ even},\\
\int_{\S^N}( uv + \langle\nabla_g\Delta_g^{(m-1)/2}u,\nabla\Delta_g^{(m-1)/2}v\rangle_g) \;dV_g,  & m \text{ odd},
\end{cases}
\end{equation}
where $\nabla_g$ is the gradient and $\Delta_g$ is the Laplace-Beltrami operator on $\sn$. Consider the elliptic operator of order $2m$ on $\sn$ given by
\begin{equation}\label{Eq:JGM-Operator}
\sP^m_g:=\prod_{k=1}^m\left( -\Delta_g + c_k  \right),\qquad c_k:=\frac{(N-2k)(N+2k-2)}{4}.
\end{equation}
This is a conformal operator. For $m=1$ it is the conformal Laplacian and for $m=2$ it is the Paneitz operator. It yields an inner product 
\begin{align}\label{norm}
\langle u,v\rangle_\ast = \int_{\S^N}u\sP^m_gv\; \;dV_g,\quad u,v\in\mathcal{C}^\infty(\S^N),
\end{align}
and the induced norm $\Vert \cdot \Vert_\ast$ is equivalent to the standard norm given by \eqref{Eq:Standard Norm}, see \cite{BaScWe,Ro}.

The stereographic projection $\sigma:\S^N\smallsetminus\{p_0\}\rightarrow\rn$ from the north pole $p_0$ is a conformal diffeomorphism and the coordinates of the standard metric $g$ in the chart given by $\sigma^{-1}$ are
\[
g_{ij} =\psi^{4/(N-2m)}\delta_{ij},
\]
where $\delta_{ij}$ is the Kronecker delta and $\psi\in D^{m,2}(\R^N)$ is
\[
\psi(x) := \left[  \frac{2}{ 1 + |x|^2 } \right]^{\frac{N-2m}{2}}.
\]
As the operators $\sP^m_{g}$ and $(-\Delta)^m$ are conformally invariant, the stereographic projection yields the relation
\begin{equation}\label{eq:equivalent_operators1}
\sP^m_{g}(u) = \psi^{1-2_m^\ast}(-\Delta)^m [\iota(u)],\qquad\text{where \ }\iota(u):=\psi(u\circ\sigma^{-1}),
\end{equation}
for every $u\in \mathcal{C}^\infty(\S^N)$, see \cite{Ro}. 

\begin{proposition}\label{prop:equivalent_spaces1}
The map 
$$\iota:(H_g^{m}(\S^N),\|\cdot\|_{*})\rightarrow (D^{m,2}(\R^N),\|\cdot\|),\qquad u\mapsto\iota(u):=\psi(u\circ\sigma^{-1}),$$ 
is an isometric isomorphism with inverse $\iota^{-1}v = \frac{1}{\psi\circ\sigma}\,v\circ\sigma$.
\end{proposition}

\begin{proof} 
As $dV_{g}=\psi^{2_m^\ast}\;dx$, we derive from \eqref{eq:equivalent_operators1} that
			\[
		\langle u_1,u_2 \rangle_{\ast}=\int_{\S^N} u_1\sP_{g} u_2\; dV_{g} = \int_{\R^N}\iota(u_1)(-\Delta)^m[\iota(u_2)] \;dx=\langle \iota(u_1),\iota(u_2) \rangle
		\]
for any $u_1,u_2\in\cC^\infty(\sn)$. The proposition now follows by density.
\end{proof}

Set $\Gamma:=O(n_1)\times O(n_2)$, where $n_1,n_2\in\n$ with $n_1,n_2\geq 2$ and $n_1+n_2=N+1$. Then $\Gamma$ acts by linear isometries on the Sobolev spaces $H^{m}_g(\sn)$ and $D^{m,2}(\rn)$ as follows.

\begin{proposition}\label{prop:sobolev_action}
For every $\gamma\in O(N+1)$,
$$\gamma:(H^{m}_g(\sn),\|\cdot\|_*)\to (H_g^m(\sn),\|\cdot\|_*),\qquad\gamma u:=u\circ\gamma^{-1},$$
and
$$\gamma:D^{m,2}(\rn)\to D^{m,2}(\rn),\qquad \gamma v:=|\det \tilde{\gamma}'|^{1/2_m^\ast} v\circ\tilde{\gamma},$$
with $\tilde\gamma:=\sigma\circ\gamma^{-1}\circ\sigma^{-1}$, are linear isometries.
\end{proposition}

\begin{proof}
	The operator $\sP_g^m$ is natural in the sense that it is invariant under changes of coordinates \cite{FeGr,Ro}. This implies, in particular, that $\gamma^\ast \sP^m_{g}=\sP^m_{g}\circ\gamma^\ast$ for every isometry $\gamma:\S^N\rightarrow \S^N$, where $\gamma^\ast$ denotes the pullback of tensors, see \cite{BaJu,Ro}. Therefore, if $u\in\cC^\infty(\sn)$ and $\gamma\in O(N+1),$
	\[
	\sP^m_g(u\circ\gamma) = (\sP^m_g\circ\gamma^\ast)(u)=(\gamma^\ast\circ \sP^m_g)(u)=\sP^m_g(u)\circ\gamma.
	\]
Then, for $u\in\mathcal{C}^\infty(\sn)$,
	\[
	\|\gamma u\|_{*}^2 =\int_{\sn}(u\circ\gamma^{-1})\sP^m_g(u\circ\gamma^{-1}) \;dV_g =\int_{\sn}(u\circ\gamma^{-1})\sP^m_g(u)\circ\gamma^{-1} \;dV_g = \int_{\sn} u\sP^m_gu \;dV_g =\Vert u\Vert^2_{*}.
	\]
This shows, by density, that $\gamma:(H^{m}_g(\sn),\|\cdot\|_*)\to (H^m_g(\sn),\|\cdot\|_*)$ is a linear isometry.

By Proposition \ref{prop:equivalent_spaces1}, the composition $\iota\circ\gamma\circ\iota^{-1}:D^{m,2}(\R^N)\rightarrow D^{m,2}(\R^N)$ is a linear isometry for every $\gamma\in\Gamma$. So $\gamma v := (\iota\circ\gamma\circ\iota^{-1})v$ defines a linear action of $\Gamma$ on $D^{m,2}(\R^N)$. Setting $\tilde\gamma:=\sigma\circ\gamma^{-1}\circ\sigma^{-1}$, we have that
\[
\gamma v = \frac{\psi}{\psi\circ\widetilde{\gamma}} v\circ\widetilde{\gamma} 
= \vert \det \widetilde{\gamma}'\vert^{1/2_m^\ast} v\circ\widetilde{\gamma}
\]
for any $\gamma\in \Gamma$ and any $v\in D^{m,2}(\R^N)$, see identity (3.2) in \cite{cp}.
\end{proof}

Define
\begin{align*}
H^{m}_g(\sn)^\Gamma:=&\{u\in H^{m}_g(\sn):\gamma u=u\text{ for all }\;\gamma\in\Gamma\}, \\
D^{m,2}(\rn)^\Gamma:=&\{v\in D^{m,2}(\rn):\gamma v=v \text{ for all }\;\gamma\in\Gamma\}.
\end{align*}
Note that the map $\iota$ from Proposition \ref{prop:equivalent_spaces1} yields an isometric isomorphism
\begin{equation} \label{eq:iota}
\iota:H^{m}_g(\sn)^\Gamma\to D^{m,2}(\rn)^\Gamma.
\end{equation}
Let $ L_g^{2^*_m}(\S^N)$ and $L^{2^*_m}(\R^N)$ denote the usual Lebesgue spaces. The crucial role played by the symmetries is given by the following statement. 

\begin{lemma}\label{Lemma:Sobolev}
	The embeddings
	\[
	H^m_g(\S^N)^\Gamma\hookrightarrow L_g^{2^*_m}(\S^N)\qquad\text{and}\qquad D^{m,2}(\R^N)^\Gamma\hookrightarrow L^{2^*_m}(\R^N)
	\]
	are continuous and compact.
\end{lemma}

\begin{proof}
The statement for $\sn$ follows from \cite[Lemma 3.2]{BaScWe}. The statement for $\rn$ is obtained using the isometry \eqref{eq:iota} and noting that $\iota:L_g^{2^*_m}(\S^N)\to L^{2^*_m}(\R^N)$ is also an isometry.
\end{proof}

To study the regularity of functions belonging to $H^{m}_g(\sn)^\Gamma$ and $D^{m,2}(\rn)^\Gamma$ we turn our attention to the space of $\Gamma$-orbits of $\sn$. 

We write $\mathbb R^{N+1}\equiv\R^{n_1}\times\R^{n_2}$. Accordingly, points in $\sn$ are written as $(x,y)\in\R^{n_1}\times\R^{n_2}$. Let $q:\S^N\rightarrow[0,\pi]$ be given by
\begin{equation} \label{eq:orbit_map}
q:=\arccos\circ f,\qquad \text{where \ }f(x,y):=|x|^2-|y|^2.
\end{equation}
This is a quotient map identifying each $\Gamma$-orbit in $\S^N$ with a single point. It is called the $\Gamma$-\textit{orbit map} of $\S^N$. Note that the $\Gamma$-orbit space of $\S^N$ is one-dimensional and that
$$q^{-1}(0)\cong\mathbb{S}^{n_1-1},\qquad q^{-1}(t)\cong\mathbb{S}^{n_1-1}\times\mathbb{S}^{n_2-1}\text{ if  }t\in(0,\pi),\qquad q^{-1}(\pi)\cong\mathbb{S}^{n_2-1}.$$

Let $\phi:(0,\pi)\rightarrow\R$ be given by 
\begin{align}\label{H}
\phi(t):=\frac{2}{\sin t}[(n_1+n_2-2)\cos t - (n_2-n_1)] 
\end{align}
and define $\sL:\mathcal{C}^\infty(0,\pi)\rightarrow\mathcal{C}^\infty(0,\pi)$ by
\begin{align*}
	\sL:=4\frac{d^2}{\;dt^2} + \phi(t)\frac{d}{\;dt}.
\end{align*}
Set
\begin{align}\label{h}
h(t):=2 \vert\S^{n_1-1}\vert \vert\S^{n_2-1}\vert \cos^{n_1-1}(t/2)\sin^{n_2-1}(t/2),\quad t\in[0,\pi],
\end{align}
where $\vert\S^{n_i-1}\vert$ is the $(n_i-1)$-dimensional measure of the sphere $\S^{n_i-1}$ for $i=1,2$.  For $\bf k=(k_0,\ldots,k_m)\in (0,\infty)^{m+1}$ and $w\in\mathcal{C}^\infty(0,\pi)$ define
\begin{equation}\label{h:norm}
	\|w\|_{\bf k,h}:=\left(\sum_{\substack{i=0\\ i\ \textrm{even}}}^m \frac{k_i}{4}\int_0^\pi |\sL^{i/2} w|^2 \;h \;dt+ \sum_{\substack{i=1\\ i\ \textrm{odd}}}^m k_i \int_0^\pi |(\sL^{(i-1)/2}w)'|^2 \; h\;dt\right)^{1/2}, 
\end{equation}
where $\sL^{i}$ denotes the $i$-fold composition of $\sL$ and $(\sL^{i}w)':=\frac{d}{dt}\Big((4\frac{d^2}{\;dt^2} + \phi(t)\frac{d}{\;dt})^{i}(w)\Big)$.

Note that the operator $\sP_g^m$ can be written as
\begin{align*}
\sP_g^m=\sum_{i=0}^ma_i(-\Delta_g)^{i} 
\end{align*}
for some $a_i>0$. Given $\bf k:=(k_0,\ldots,k_m)\in(0,\infty)^{m+1}$, we consider the operator
\begin{equation*}
\sP_{\bf k,g}^m:=\sum_{i=0}^m k_i(-\Delta_g)^{i},
\end{equation*}
and the norm
\begin{equation}\label{k}
\|u\|_{\bf k,\ast}:=\left(\int_{\S^N}u\sP_{\bf k,g}^m u \ \;dV_g\right)^\frac{1}{2}\qquad \text{ for }u\in \cC^\infty(\S^N).
\end{equation}
Note that $\|\cdot\|_*=\|\cdot\|_{\bf a,*}$ with $\bf a=(a_0,\ldots,a_m)$ as above. So $\|\cdot\|_{\bf k,*}$ is equivalent to $\|\cdot\|_*$.

\begin{lemma} \label{lem:isometry2} 
For every $\bf k\in(0,\infty)^{m+1}$ and $w\in\mathcal{C}^\infty[0,\pi]$,
$$\|w\circ q\|_{\bf k,*}^2= \|w\|_{\bf k,h}^2.$$
\end{lemma}

\begin{proof}
Set $u:=w\circ q$. For $u_1,u_2\in \cC^\infty(\S^N)$, observe that, if $i$ is even, then
\[
\int_{\S^N} u_1 (-\Delta_g)^{i}u_2\;dV_g = \int_{\S^N} \Delta_g^{i/2}u_1 \Delta_g^{i/2} u_2 \; \;dV_g,
\]
while, if  $i$ is odd,
\[
\int_{\S^N} u_1 (-\Delta_g)^{i}u_2 \;dV_g
= \int_{\S^N} \langle\nabla_g \Delta_g^{(i-1)/2}u_1,\nabla_g \Delta_g^{(i-1)/2} u_2\rangle_g \; \;dV_g.
\]
Hence, 
\begin{equation}
\begin{split}
\|u\|_{\bf k,\ast}^2=&\sum_{\substack{i=0\\ i\ \textrm{even}}}^m k_i\int_{\S^N} |\Delta^{i/2}u|^2 \; \;dV_g + \sum_{\substack{i=0\\ i\ \textrm{odd}}} k_i  \int_{\S^N} |\nabla_g \Delta_g^{(i-1)/2}u|_g^2 \; \;dV_g.
\end{split}
\end{equation}
Note that, for the function $f$ defined in \eqref{eq:orbit_map}, the sets $M_+:= f^{-1}(1)$ and $M_-:= f^{-1}(-1)$ are submanifolds of $\S^{N}$ diffeormorphic to $\S^{n_1-1}$ and $\S^{n_2-1}$ respectively. As in \cite{fp}, we have that
 \[
 \vert \nabla_g f\vert_g^2 = 4(1 - f^2)\quad\text{ and } \quad\Delta_g f= -2(N+1)f + 2(n_2-n_1).
 \]
Then, by the definition of $q$,
\[
|\nabla_g q|_g^2 = 4\qquad\text{and}\qquad \Delta_g q=\phi\circ q,
\]
so
\[
\Delta_{g}u=\Delta_g(w\circ q)=(w''\circ q) |\nabla_g q|_g^2+(w'\circ q)\Delta_gq=(\sL w)\circ q, \quad \text{ in }\S^N\smallsetminus M_+\cup M_-
\]
and, for each $i\in\N\cup\{0\}$,
\[
\Delta_g^i u = (\sL^i w)\circ q \quad \text{and}\quad |\nabla_g \Delta_g^i u|_g^2 = 4|(\sL^i w_1)'|^2\circ q, \quad \text{ in }\S^N\smallsetminus M_+\cup M_-. 
\]
By \cite[Lemma 2.2]{fp},
\begin{equation*}
\int_{\S^N}|\Delta_g^{i}u|^2 \; \;dV_g 
= \frac{1}{4}\int_0^\pi |\sL^i w|^2 \; h\;dt\label{Eq:Norm even}\quad \text{and}\quad
\int_{\S^N}|\nabla_g\Delta_g^{i}u|_g^2 \; \;dV_g  = \int_0^\pi |(\sL^{i}w)'|^2 \;h \;dt.
\end{equation*} 
Therefore,
\begin{equation*}
	\|u\|^2_{\bf k,*}=\sum_{\substack{i=0\\ i\ \textrm{even}}}^m \frac{k_i}{4}\int_0^\pi |\sL^{i/2} w|^2 \;h \;dt+ \sum_{\substack{i=1\\ i\ \textrm{odd}}}^m k_i \int_0^\pi |(\sL^{(i-1)/2}w)'|^2 \;h \;dt=\|w\|^2_{\bf k,h},
\end{equation*}
as claimed.
\end{proof}

For $\varepsilon>0$, let $H^{m}(\varepsilon,\pi-\varepsilon)$ denote the usual Sobolev space of order $m$ in the interval $(\varepsilon,\pi-\varepsilon)$.

\begin{lemma}\label{p}
For each $\varepsilon>0$, there are $\bf k=(k_0,\ldots, k_m)\in(0,\infty)^{m+1}$ and $A>0$, depending on $\eps$, such that
\[
\Vert w\Vert_{\bf k,h}\geq A \Vert w\Vert_{H^{m}(\varepsilon,\pi-\varepsilon)}\qquad \text{for every \  }w\in\mathcal{C}^\infty[0,\pi].
\]
\end{lemma}

\begin{proof}
By Lemma \ref{bb:lemma}, for every $\varepsilon>0$ there are $\eta>0$ and $\mu>1$, depending on $\eps$, such that, for $i\geq 2$ even
\begin{align}\label{i1}
\frac{1}{4}\,|\sL^{i/2} w|^2 h&\geq \eta\left(|w^{(i)}|^2 - \mu\sum_{j=1}^{i - 1}|w^{(j)}|^2\right)\quad \text{ in }(\varepsilon,\pi-\varepsilon),
\end{align}
and for $i$ odd
\begin{align}\label{i2}
|(\sL^{(i-1)/2} w)'|^2 h &\geq \eta
\left(|w^{(i)}|^2 -\mu\sum_{j=1}^{i-1}|w^{(j)}|^2\right)\quad \text{ in }(\varepsilon,\pi-\varepsilon).
\end{align}
Let $k_0:=1$, $k_i:=(2\mu)^{-i}$ for $i\geq 1$, and $\bf k:=(k_0,\ldots, k_m)\in(0,\infty)^{m+1}$. By \eqref{i1}, \eqref{i2},
\begin{align*}
 &\Vert w\Vert^2_{\bf k,h}
 =\sum_{\substack{i=0\\ i\ \textrm{even}}}^m \frac{k_i}{4}\int_0^\pi |\sL^{i/2} w|^2 h\; \;dt
 + \sum_{\substack{i=1\\ i\ \textrm{odd}}}^m k_i \int_0^\pi |\sL^{(i-1)/2}w'|^2 h\; \;dt\\
 &\geq \eta\int_\varepsilon^{\pi-\varepsilon}
 \Big[\Big(\sum_{i=0}^m k_i|w^{(i)}|^2\Big)
 -\mu\Big(\sum_{i=0}^m k_i\sum_{j=1}^{i - 1}|w^{(j)}|^2\Big)\Big]\; \;dt\\
  &= \eta\int_\varepsilon^{\pi-\varepsilon}\Big[
  k_0|w|^2 
  +\Big(k_1- \mu \sum_{i=2}^mk_i\Big)|w^{(1)}|^2 
  +\sum_{i=2}^{m-1}\Big(k_i- \mu \sum_{j=i+1}^m k_{i}\Big)|w^{(i)}|^2 + k_m \vert w^{(m)}\vert^2\Big] \;dt\\
  &= \eta\int_\varepsilon^{\pi-\varepsilon}\Big[
  |w|^2 
  +\Big(\frac{1}{2\mu}-\sum_{i=2}^m \frac{1}{2^i\mu^{i-1}}\Big)|w^{(1)}|^2
  +\sum_{i=2}^{m-1}\Big(\frac{1}{2^i\mu^i}-\sum_{j=i+1}^m \frac{1}{2^j\mu^{j-1}}\Big)|w^{(i)}|^2 +(2\mu)^{-m} \vert w^{(m)}\vert^2\Big] \;dt.
 \end{align*}
 For $i=1,\ldots,m-1$,
 \begin{align*}
  A_i:=\frac{1}{2^i\mu^i}-\sum_{j=i+1}^m \frac{1}{2^j\mu^{j-1}}=\frac{2^{-i} (\mu -1) \mu ^{-i}+2^{-m} \mu ^{1-m}}{2 \mu -1}>0,
 \end{align*}
and the claim follows with $A:=\eta\min\{A_1,\ldots,A_{m-1},(2\mu)^{-m}\}>0$.
\end{proof}

We have the following regularity result.

\begin{proposition} \label{prop:continuity}
Let $Z:=(\S^{n_1-1}\times\{0\})\, \cup\, (\{0\}\times\S^{n_2-1})\subset\S^N$. For every $u\in H^m_g(\S^N)^\Gamma$ there exists $\widetilde u\in \cC^{m-1}(\S^N\smallsetminus Z)^\Gamma$ such that $u=\widetilde u$ a.e. in $\S^N$.
\end{proposition}

\begin{proof}
Fix $\eps>0$ and let $\Theta_\eps:=q^{-1}(\eps,\pi-\eps)$. Then, by Lemmas \ref{lem:isometry2} and \ref{p}, there exists $C>0$, depending on $\eps$, such that $\|u\|_*\geq C\|w\|_{H^m(\eps,\pi-\eps)}$ because the norm defined in \eqref{k} is equivalent to $\|\cdot\|_*$.
Therefore, the map
$$H^{m}_g(\Theta_\eps)^\Gamma\to H^m(\eps,\pi-\eps),\qquad u\mapsto w,\quad\text{ \ where \ }u=w\circ q,$$
is continuous. Sobolev's theorem yields a continuous embedding $H^m(\eps,\pi-\eps)\hookrightarrow\cC^{m-1}(\eps,\pi-\eps)$. Thus, for $u\in H^m_g(\S^N)^\Gamma$ and $w$ given by $u=w\circ q$, there exists $w_\eps\in\cC^{m-1}(\eps,\pi-\eps)$ such that $w=w_\eps$ a.e. in $(\eps,\pi-\eps)$. So $u_\eps:=w_\eps\circ q\in\cC^{m-1}(\Theta_\eps)$ and $u=u_\eps$ a.e. in $\Theta_\eps$. The function $\tilde u(p):=u_\eps(p)$ if $p\in\Theta_\eps$ is well defined and of class $\cC^{m-1}$ on $\sn\smallsetminus Z$, and it coincides a.e. with $u$.
\end{proof}

\begin{remark}\label{b:rmk}
\emph{
Let $u\in H^m_g(\S^N)^\Gamma$. Since $u$ is $\Gamma$-invariant, there is $w:[0,\pi]\to \R$ such that $u=w\circ q$, with $q$ as in \eqref{eq:orbit_map}. As a consequence, problem \eqref{eq:problem} can be seen as an ODE.}

\emph{
In particular, if $m=1$, Lemma~\ref{lem:isometry2} yields that
\begin{align*}
\|u\|_*^2=\int_0^\pi \left(|w'(t)|^2 + \frac{c_1}{4}\,  w(t)^2\right)h(t)\ dt,
\end{align*}
where the constant $c_1$ is as in \eqref{Eq:JGM-Operator}. As a consequence, $u\in H^m_g(\S^N)^\Gamma$ is a solution to the Yamabe equation \eqref{eq:problem} with $m=1$ iff $w$ solves the ODE
\begin{align*}
 -(w'h)'+ \frac{c_1}{4}\,  w\, h=-w''h-w'h'+ \frac{c_1}{4}\,  w\, h = \frac{h}{4}|w|^{2^*_1-2}w\quad \text{ in }(0,\pi).
\end{align*}
A careful study of this ODE is performed in \cite{fp} to obtain existence of solutions to the Yamabe equation on the sphere with exactly $\ell$-nodal regions for any $\ell\in\mathbb{N}$.  A similar analysis is much harder for $m\geq 2$, where the coefficients of the ODE are more complex. For instance, if  $u\in H^2_g(\S^N)^\Gamma$ is a solution of \eqref{eq:problem} with $m=2$ and $w:[0,\pi]\to \R$ is such that $u=w\circ q$, then, by Lemma~\ref{lem:isometry2},
\begin{align*}
 \|u\|_*^2&=\|w\|^2_{(a_0,a_1,1),h}
 =\frac{a_0}{4}\int_0^\pi |w|^2 h\; \;dt
 + a_1 \int_0^\pi |w'|^2 h\; \;dt
 +\frac{1}{4}\int_0^\pi |
 4w'' + \phi(t)w'|^2 h\; \;dt\\
  &=\int_0^\pi 
 \left(
4 w''(t)^2
+\left(\frac{1}{4} \phi(t)^2+a_1\right) w'(t)^2
+ 2 \phi(t) w'(t) w''(t) 
   +\frac{a_0}{4} w(t)^2
   \right)h(t)  \ \mathrm{d}t,
\end{align*}
where $a_0=c_1c_2$, $a_1=c_1+c_2$, and $c_1,c_2$ are given in \eqref{Eq:JGM-Operator}. The associated fourth-order ODE for \eqref{eq:problem} with $m=2$ is
\begin{align*}
&4 h\, w''''+ 8 h'\, w'''+C_1\, w''+C_2\, w'+ \frac{a_0}{4} h\, w(t)=\frac{h}{4}|w|^{2^*_2-2}w\quad \text{ in }(0,\pi),
\end{align*}
where
\begin{align*}
 C_1(t)&:=4 h''(t)+2 \phi(t) h'(t)+2 h(t)\phi'(t)-\frac{1}{4} h(t) \phi(t)^2 - a_1 h(t),\\
 C_2(t)&:=4 h'(t) \phi'(t)+ 2 \phi(t) h''(t)- \frac{1}{4} \phi(t)^2 h'(t)- a_1 h'(t)+2 h(t) \phi''(t)- \frac{1}{2} h(t) \phi(t) \phi'(t).
\end{align*}}
\end{remark}

\section{The polyharmonic system}
\label{sec:system}

We fix $\Gamma:=O(n_1)\times O(n_2)$ with $n_1,n_2\geq 2$ and $n_1+n_2=N+1$ and we study the system \eqref{eq:system}. Let $\cH:=(D^{1,2}(\rn)^\Gamma)^\ell$ with the norm
\begin{equation*}\label{Eq:NormProduct}
\Vert \bar{u}\Vert =\Vert(u_1,\ldots,u_\ell)\Vert = \Big(\sum_{i=1}^\ell \Vert u_i \Vert^2\Big)^{1/2},
\end{equation*}
and $\mathcal{J}:\mathcal{H}\rightarrow\R$ be the functional given by
\[
\mathcal{J}(\bar{u}):=\frac{1}{2}\sum_{i=1}^\ell\Vert u_i\Vert^2 - \frac{1}{2^*_m}\sum_{i=1}^\ell\int_{\R^N}\mu_i\vert u_i\vert^{2^*_m} - \frac{1}{2}\sum_{\substack{i,j=1 \\ j\neq i}}^\ell\int_{\R^N}\lambda_{ij}\vert u_j\vert^{\alpha_{ij}}\vert u_i\vert^{\beta_{ij}}.
\]
This is a $\cC^1$-functional and, by the principle of symmetric criticality \cite{palais}, its critical points are the solutions of \eqref{eq:system}. Observe that the fully nontrivial critical points of $\cJ$ belong to the set
\[
\mathcal{N}:=\{\bar{u}\in\mathcal{H}\;:\; u_i\neq 0,\ \partial_i\mathcal{J}(\bar{u})u_i = 0, \text{ for each }i=1,\ldots,\ell\}.
\]
Note also that, for each $i$,
\begin{align*}
&\partial_i\mathcal{J}(\bar u)u_i=\|u_i\|^2 - \irn \mu_i|u_i|^{2_m^*} - \sum_{\substack{j=1 \\ j\neq i}}^\ell\irn\lambda_{ij}\beta_{ij}|u_j|^{\alpha_{ij}}|u_i|^{\beta_{ij}}.
\end{align*}
It is readily seen that 
\begin{equation} \label{eq:energy_nehari}
\cJ(\bar u)=\frac{m}{N}\|\bar u\|^2\qquad\text{if \ }\bar u\in\cN.
\end{equation}

\begin{lemma} \label{lem:away_from_0}
There exists $d_0>0$, independent of $\lambda_{ij}$, such that $\min_{i=1,\ldots,\ell}\|u_i\|\geq d_0$ if $\bar u=(u_1,\ldots,u_\ell)\in \cN$. Thus, $\cN$ is a closed subset of $\cH$ and $\inf_\cN\cJ>0$.
\end{lemma}

\begin{proof}
From $\lambda_{ij}<0$ and Sobolev's inequality we obtain
\begin{align*}
\|u_i\|^2\leq \irn \mu_i|u_i|^{2_m^*}\leq C\|u_i\|^{2_m^*}\quad \text{ for \ }\bar u\in \cN, \ i=1,\ldots,\ell,
\end{align*}
with $C>0$.
\end{proof}

\begin{definition} \label{def:least energy}
A fully nontrivial solution $\bar u$ to the system \eqref{eq:system} satisfying $\cJ(\bar u)=\inf_\cN\cJ$ is called a least energy solution.
\end{definition}

To establish the existence of fully nontrivial critical points of $\cJ$ we follow the variational approach introduced in \cite{cs}. 

Given $\bar{u}=(u_1,\ldots,u_\ell)$ and $\bar{s}=(s_1,\ldots,s_\ell)\in(0,\infty)^\ell$, we write
\[
\bar{s}\bar{u}:= (s_1u_1,\ldots,s_\ell u_\ell).
\]
Let $\mathcal S:=\{u\in D^{1,2}(\rn)^\Gamma:\|u\|=1\}$, $\cT:=\mathcal S^\ell$, and define
$$\cU:=\{\bar{u}\in\cT:\bar{s}\bar{u}\in\cN\text{ \ for some \ }\bar s\in(0,\infty)^\ell\}.$$

\begin{lemma} \label{lem:U}
\begin{itemize}
\item[$(i)$] Let $\bar u\in\cT$. If there exists $\bar s_{\bar u}\in(0,\infty)^\ell$ such that $\bar s_{\bar u}\bar u\in\cN$, then $\bar s_{\bar u}$ is unique and satisfies
$$\cJ(\bar s_{\bar u}\bar u)=\max_{\bar s\in(0,\infty)^\ell}\cJ(\bar s\bar u).$$
\item[$(ii)$] $\cU$ is a nonempty open subset of $\cT$, and the map $\cU\to(0,\infty)^\ell$ given by $\bar u\mapsto\bar s_{\bar u}$ is continuous.
\item[$(iii)$] The map $\cU\to \cN$ given by $\bar u\mapsto\bar s_{\bar u}\bar u$ is a homeomorphism.
\item[$(iv)$] If $(\bar u_n)$ is a sequence in $\cU$ and $\bar u_n\to\bar u\in\partial\cU$, then $|\bar s_{\bar u_n}|\to\infty$.
\end{itemize}
\end{lemma}

\begin{proof}
The same arguments used in the proof of \cite[Proposition 3.1]{cs} give the proof of this result.
\end{proof}

Define $\Psi:\cU\to\r$ as 
\begin{equation*}
\Psi(\bar u): = \cJ(\bar s_{\bar u}\bar u).
\end{equation*}
According to Lemma \ref{lem:U}, $\cU$ is an open subset of the smooth Hilbert submanifold $\cT$ of $\cH$. If $\Psi$ is of class $\cC^1$ we write $\|\Psi'(\bar u)\|_*$ for the the norm of $\Psi'(\bar u)$ in the cotangent space $\mathrm{T}_{\bar u}^*(\cT)$ to $\cT$ at $\bar u$, i.e.,
$$\|\Psi'(\bar u)\|_*:=\sup\limits_{\substack{\bar v\in\mathrm{T}_{\bar u}(\cU) \\\bar v\neq 0}}\frac{|\Psi'(\bar u)\bar v|}{\|\bar v\|},$$
where $\mathrm{T}_{\bar u}(\cU)$ is the tangent space to $\cU$ at $\bar u$.

Recall that a sequence $(\bar u_n)$ in $\cU$ is called a $(PS)_c$\emph{-sequence for} $\Psi$ if $\Psi(\bar u_n)\to c$ and $\|\Psi'(\bar u_n)\|_*\to 0$, and $\Psi$ is said to satisfy the $(PS)_c$\emph{-condition} if every such sequence has a convergent subsequence. Similarly, a $(PS)_c$\emph{-sequence for} $\cJ$ is a sequence $(\bar u_n)$ in $\cH$ such that $\cJ(\bar u_n)\to 0$ and $\|\cJ'(\bar u_n)\|_{\cH'}\to 0$, and $\cJ$ satisfies the $(PS)_c$\emph{-condition} if any such sequence has a convergent subsequence.   Here $\cH'$ denotes, as usual, the dual space of $\cH$.

\begin{lemma} \label{lem:psi}
\begin{itemize}
\item[$(i)$] $\Psi\in\cC^1(\cU,\r)$,
\begin{equation*}
\Psi'(\bar u)\bar v = \cJ'(\bar s_{\bar u}\bar u)[\bar s_{\bar u}\bar v] \quad \text{for all } \bar u\in\cU \text{ and }\bar v\in \mathrm{T}_{\bar u}(\cU),
\end{equation*}
and there exists $d_0>0$ such that
$$d_0\,\|\cJ'(\bar s_{\bar u}\bar u)\|_{\cH'}\leq\|\Psi'(\bar u)\|_*\leq |\bar s_{\bar u}|_\infty\|\cJ'(\bar s_{\bar u}\bar u)\|_{\cH'}\quad \text{for all } \bar u\in\cU,$$
where $|\bar s|_\infty=\max\{|s_1|,\ldots,|s_q|\}$ if $\bar s=(s_1,\ldots,s_q)$.
\item[$(ii)$] If $(\bar u_n)$ is a $(PS)_c$-sequence for $\Psi$, then $(\bar s_{\bar u_n}\bar u_n)$ is a $(PS)_c$-sequence for $\cJ$.
\item[$(iii)$] $\bar u$ is a critical point of $\Psi$ if and only if $\bar s_{\bar u}\bar u$ is a critical point of $\cJ$ if and only if $\bar s_{\bar u}\bar u$ is a fully nontrivial solution of \eqref{eq:system}.
\item[$(iv)$] If $(\bar u_n)$ is a sequence in $\cU$ and $\bar u_n\to\bar u\in\partial\cU$, then $|\Psi(\bar u_n)|\to\infty$.
\item[$(v)$]$\bar{u}\in\cU$ if and only if $-\bar{u}\in\cU$, and $\Psi(\bar u)=\Psi(-\bar u)$.
\end{itemize}
\end{lemma}

\begin{proof}
These statements are proved arguing exactly as in \cite[Theorem 3.3]{cs}.
\end{proof}

\begin{lemma}
$\Psi$ satisfies the $(PS)_c$-condition for every $c\in\r$.
\end{lemma}

\begin{proof}
Let $(\bar{v}_n)$ be a $(PS)_c$-sequence for $\mathcal{J}$ with $\bar v_n\in\cN$. Then
\[
	\frac{m}{N}\Vert \bar{v}_n \Vert^2 = \mathcal{J}(\bar{v}_n) - \frac{1}{2_m^\ast} \mathcal{J}'(\bar{v}_n)\bar{v}_n \leq c(1 + \Vert \bar{v}_n\Vert)
\]
for some positive constant $c$ not depending on $\bar{v}_n$, so the sequence is bounded. A standard argument using Lemma \ref{Lemma:Sobolev}, as in \cite[Proposition 3.6]{cp}, shows that $(\bar{v}_n)$ contains a convergent subsequence. The statement of the lemma follows from Lemmas \ref{lem:psi}$(ii)$ and \ref{lem:U}$(iii)$.
\end{proof}

Given a nonempty subset $\mathcal{Z}$ of $\mathcal{T}$ such that $\bar{u}\in\mathcal{Z}$ if and only if $-\bar{u}\in\mathcal{Z}$, the \emph{genus of $\mathcal{Z}$}, denoted $\mathrm{genus}(\mathcal{Z})$, is the smallest integer $k\geq 1$ such that there exists an odd continuous function $\mathcal{Z}\rightarrow\mathbb{S}^{k-1}$ into the unit sphere $\mathbb{S}^{k-1}$ in $\R^k$. If no such $k$ exists, we define $\mathrm{genus}(\mathcal{Z})=\infty$; finally, we set $\mathrm{genus}(\emptyset)=0$.

\begin{lemma}
$\mathrm{genus}(\cU)=\infty$.
\end{lemma}

\begin{proof}
As in \cite[Lemma 3.2]{cp} one constructs $\Gamma$-invariant functions in $\cC^\infty(\rn)$ with disjoint supports. Then, arguing as in \cite[ Lemma 4.5]{cs}, one shows that $\mathrm{genus}(\cU)=\infty$.
\end{proof}
\smallskip

\begin{proof}[Proof of Theorem \ref{thm:existence}]
Lemma \ref{lem:psi}$(iv)$ implies that $\cU$ is positively invariant under the negative pseudogradient flow of $\Psi$, so the usual deformation lemma holds true for $\Psi$, see e.g. \cite[Section II.3]{s} or \cite[Section 5.3]{w}. As $\Psi$ satisfies the $(PS)_c$-condition for every $c\in\r$, standard variational arguments show that $\Psi$ attains its minimum on $\cU$ at some $\bar u$. By Lemma \ref{lem:psi}$(iii)$ and the principle of symmetric criticality, $\bar s_{\bar u}\bar u$ is a least energy fully nontrivial solution of the system \eqref{eq:system}. Moreover, as $\Psi$ is even and $\mathrm{genus}(\cU)=\infty$, $\Psi$ has an unbounded sequence of critical values. Since $\Psi(\bar u)=\cJ(\bar s_{\bar u}\bar u)=\frac{m}{N}\|\bar s_{\bar u}\bar u\|^2$ by \eqref{eq:energy_nehari}, the system \eqref{eq:system} has an unbounded sequence of fully nontrivial solutions.
\end{proof}

\section{Segregation and optimal partitions}
\label{sec:segregation}

Let $\Gamma$ be as before and let $\Omega$ be a $\Gamma$-invariant open subset of $\mathbb{R}^N$. The solutions to the problem \eqref{eq:dirichlet} are the critical points of the energy functional $J_\Omega: D_0^{m,2}(\Omega)^\Gamma\rightarrow\R$ defined by
\begin{align*}
J_\Omega(v):=\frac{1}{2}  \|v\|^2-\frac{1}{{2^*_m}}\int_\Omega|v|^{2^*_m}.
\end{align*}
The nontrivial ones belong to the Nehari manifold
\begin{align*}
	\cM_\Omega:=&\{v\in D^{m,2}_0(\Omega)^\Gamma:v\neq 0,\;J_\Omega'(v)v=0\} \\
	=&\{v\in D^{m,2}_0(\Omega)^\Gamma:v\neq 0,\;\|v\|^2=\int_\Omega|v|^{2^*_m}\},
\end{align*}
which is a closed submanifold of $D^{m,2}_0(\Omega)^\Gamma$ of class $\cC^2$ and a natural constraint for $J_\o$. A minimizer for $J_\o$ on $\cM_\o$ is called a \emph{least energy $\Gamma$-invariant solution to \eqref{eq:dirichlet} in $\o$}. By standard arguments, using Lemma \ref{Lemma:Sobolev}, one sees that \eqref{eq:dirichlet} does have a least energy solution. So the quantity $c_\o^\Gamma$ defined in the introduction is
$$c_{\o}^\Gamma=\inf_{u\in\cM_\o}J_\o(u).$$

We begin by establishing some properties of optimal partitions. Let
$$\widetilde{q}:=q\circ\sigma^{-1}:\mathbb{R}^N\to[0,\pi],$$
where $\sigma$ is the stereographic projection and $q$ is the $\Gamma$-orbit map of $\S^N$ defined in \eqref{eq:orbit_map}. So, writing $\r^{N+1}=\r^{n_1}\times\r^{n_2}$, one has that $\widetilde q^{\,-1}(0) = \mathbb{S}^{n_1-1}\times\{0\}$ and $\widetilde q^{\,-1}(\pi) = \{0\}\times \r^{n_2-1}$.

\begin{lemma} \label{lem:tori}
Let $\ell\geq 2$ and $\{\Theta_1,\ldots,\Theta_\ell\}\in\mathcal{P}_\ell^\Gamma$ be a $(\Gamma,\ell)$-optimal partition for problem \eqref{p}. Then, the following statements hold true.
\begin{itemize}
\item[$(i)$] There exist $a_1,\ldots,a_{\ell-1}\in(0,\pi)$ such that
$$(0,\pi)\smallsetminus\bigcup_{i=1}^\ell\widetilde{q}\,(\Theta_i)=\{a_1,\ldots,a_{\ell-1}\}.$$
Therefore, after reordering,
\begin{align*}
\Omega_1 :=& \ \Theta_1\cup(\mathbb{S}^{n_1-1}\times\{0\})=\widetilde{q}\,^{-1}[0,a_1),\\
\Omega_i :=& \ \Theta_i=\widetilde{q}\,^{-1}(a_{i-1},a_i)\quad\text{for }\; i=2,\ldots,\ell-1,\\
\Omega_\ell :=& \ \Theta_\ell\cup(\{0\}\times\mathbb{R}^{n_2-1})=\widetilde{q}\,^{-1}(a_{\ell-1},\pi].
\end{align*}
\item[$(ii)$] $\Omega_1,\ldots,\Omega_\ell$ are smooth and connected, they satisfy items $(c_1)$ and $(c_2)$ of \emph{Theorem}~\ref{thm:main}, $\Omega_1,\ldots,\Omega_{\ell-1}$ are bounded, $\Omega_\ell$ is unbounded, $\overline{\Omega_1\cup\cdots\cup \Omega_\ell}=\mathbb{R}^{N}$, and $\{\Omega_1,\ldots,\Omega_\ell\}\in\mathcal{P}_\ell^\Gamma$ is a $(\Gamma,\ell)$-optimal partition for problem \eqref{p}.
\end{itemize}
\end{lemma}

\begin{proof}
$(i):$ Let $a,b,c\in(0,\pi)$ with $a<b<c$ and set $\Lambda_1:=\widetilde{q}\,^{-1}(a,b)$,\; $\Lambda_2:=\widetilde{q}\,^{-1}(b,c)$,\; $\Lambda=\widetilde{q}\,^{-1}(a,c)$. As $\Lambda_i\subset\Lambda$, we have that $c_{\Lambda}^\Gamma \leq \min\{c_{\Lambda_1}^\Gamma,c_{\Lambda_2}^\Gamma\}$. We claim that 
$$c_{\Lambda}^\Gamma < \min\{c_{\Lambda_1}^\Gamma,c_{\Lambda_2}^\Gamma\}. $$
Indeed, if $c_{\Lambda}^\Gamma=c_{\Lambda_1}^\Gamma$ then, taking a least energy $\Gamma$-invariant solution to \eqref{eq:dirichlet} in $\Lambda_1$ and extending it by $0$ in $\Lambda\smallsetminus\Lambda_1$ we obtain a least energy $\Gamma$-invariant solution $u$ to \eqref{eq:dirichlet} in $\Lambda$. Then, $u\in\cC^{2m}(\Lambda)$ by \cite{l} and it vanishes in $\Lambda\smallsetminus\Lambda_1$, contradicting the unique continuation principle \cite{lin,protter}.

Therefore, if $\{\Theta_1,\ldots,\Theta_\ell\}\in\mathcal{P}_\ell^\Gamma$ is a $(\Gamma,\ell)$-optimal partition for problem \eqref{p}, then $(0,\pi)\smallsetminus\bigcup_{i=1}^\ell\widetilde{q}\,(\Theta_i)$ must consist of precisely $\ell-1$ points.

$(ii):$ Clearly, $\Omega_1,\ldots,\Omega_\ell$ are smooth and connected and satisfy statements $(c_1)$ and $(c_2)$ of Theorem \ref{thm:main}. Moreover, $\Omega_1,\ldots,$ $\Omega_{\ell-1}$ are bounded, $\Omega_\ell$ is unbounded, $\mathbb{R}^{N}=\overline{\Omega_1\cup\cdots\cup \Omega_\ell}$, and $\{\Omega_1,\ldots,\Omega_\ell\}\in\mathcal{P}_\ell^\Gamma$. 

As $\Theta_i\subset\o_i$ we have that $c_{\Omega_i}^\Gamma\leq c_{\Theta_i}^\Gamma$ for all $i$. So, as $\{\Theta_1,\ldots,\Theta_\ell\}$ is a $(\Gamma,\ell)$-optimal partition, we conclude that $\{\Omega_1,\ldots,\Omega_\ell\}$ is a $(\Gamma,\ell)$-optimal partition.
\end{proof}
\smallskip

\begin{proof}[Proof of Theorem \ref{thm:main}] Fix $\mu_i=1$ in \eqref{eq:system} for each $i=1,\ldots,\ell$, and let $\lambda_{ij,k}\to-\infty$ as $k\to\infty$. To highlight the role of $\lambda_{ij,k}$, we write $\mathcal{J}_k$ and $\mathcal{N}_k$ for the functional and the set associated to the system \eqref{eq:system} with $\lambda_{ij}$ replaced by $\lambda_{ij,k}$, introduced in Section~\ref{sec:system}. Let $\bar u_k=(u_{k,1},\ldots,u_{k,\ell})\in\cN_k$ be such that
$$c_k^\Gamma:= \inf_{\mathcal{N}_k} \mathcal{J}_k =\mathcal{J}_k(\bar u_k)=\frac{m}{N}\sum_{i=1}^\ell\|u_{k,i}\|^2.$$
Let
\begin{align*}
\mathcal{N}_0:=\{(v_1,\ldots,v_\ell)\in\mathcal{H}:\,&v_i\neq 0,\;\|v_i\|^2=\irn|v_i|^{{2^*_m}}, \text{ and }v_iv_j=0\text{ a.e. in }\mathbb{R}^N \text{ if }i\neq j\}.
\end{align*}
Then, $\mathcal{N}_0\subset\mathcal{N}_k$ for all $k\in\mathbb{N}$ and, therefore, 
$$0<c_k^\Gamma\leq c_0^\Gamma:=\inf\left\{\frac{m}{N}\sum_{i=1}^\ell\|v_i\|^2:(v_1,\ldots,v_\ell)\in\mathcal{N}_0\right\}<\infty.$$
So, after passing to a subsequence, using Lemma~\ref{Lemma:Sobolev}, we get that $u_{k,i} \rightharpoonup u_{\infty,i}$ weakly in $D_{0}^{m,2}(\mathbb{R}^N)^\Gamma$, $u_{k,i} \to u_{\infty,i}$ strongly in $L^{{2^*_m}}(\mathbb{R}^N)$, and $u_{k,i} \to u_{\infty,i}$ a.e. in $\mathbb{R}^N$ for each $i=1,\ldots,\ell$.  Moreover, as $\partial_i\mathcal{J}_k(\bar u_k)[u_{k,i}]=0$, we have for each $j\neq i$,
\begin{align*}
0\leq\irn\beta_{ij}|u_{k,j}|^{\alpha_{ij}}|u_{k,i}|^{\beta_{ij}}\leq \frac{1}{-\lambda_{ij,k}}\irn|u_{k,i}|^{{2^*_m}}\leq \frac{C}{-\lambda_{ij,k}}.
\end{align*}
Then, Fatou's lemma yields 
$$0 \leq \irn|u_{\infty,j}|^{\alpha_{ij}}|u_{\infty,i}|^{\beta_{ij}} \leq \liminf_{k \to \infty} \irn|u_{k,j}|^{\alpha_{ij}}|u_{k,i}|^{\beta_{ij}} = 0.$$
Hence, $u_{\infty,j} u_{\infty,i} = 0$ a.e. in $\mathbb{R}^N$. By Lemma \ref{lem:away_from_0},
$$0<d_0 \leq \|u_{k,i}\|^2 \leq\irn |u_{k,i}|^{{2^*_m}}\qquad\text{for all \ }k\in\mathbb{N},\;i=1,\ldots,\ell,$$
and, as $u_{k,i} \to u_{\infty,i}$ strongly in $L^{{2^*_m}}(\mathbb{R}^N)$ and $u_{k,i} \rightharpoonup u_{\infty,i}$ weakly in $D^{m,2}(\mathbb{R}^N)$, we get
\begin{equation} \label{eq:comparison2}
0<\|u_{\infty,i}\|^2 \leq \irn|u_{\infty,i}|^{{2^*_m}}\qquad\text{for every \ }i=1,\ldots,\ell.
\end{equation}
 Since $u_{\infty,i}\neq 0$, there is a unique $t_i\in(0,\infty)$ such that $\|t_iu_{\infty,i}\|_0^2 = \irn|t_iu_{\infty,i}|^{{2^*_m}}$. So $(t_1u_{\infty,1},\ldots,t_\ell u_{\infty,\ell})\in \mathcal{N}_0$. The inequality \eqref{eq:comparison2} implies that $t_i\in (0,1]$. Therefore,
\begin{align*}
c_0^\Gamma &\leq \frac{m}{N}\sum_{i=1}^\ell\|t_iu_{\infty,i}\|^2 \leq \frac{m}{N}\sum_{i=1}^\ell\|u_{\infty,i}\|^2\leq \frac{m}{N}\liminf_{k\to\infty}\sum_{i=1}^\ell\|u_{k,i}\|^2=\liminf_{k\to\infty} c_k^\Gamma \leq c_0^\Gamma.
\end{align*}
It follows that $u_{k,i} \to u_{\infty,i}$ strongly in $D^{m,2}(\mathbb{R}^N)^\Gamma$ and $t_i=1$, yielding 
\begin{equation}\label{eq:limit}
\|u_{\infty,i}\|^2 = \irn|u_{\infty,i}|^{{2^*_m}},\qquad\text{and}\qquad\frac{m}{N}\sum_{i=1}^\ell\|u_{\infty,i}\|^2 =c_0^\Gamma.
\end{equation}

Set $Y_1:=\mathbb{S}^{n_1-1}\times\{0\}$, $Y_2:=\{0\}\times\mathbb{R}^{n_2-1}$, and $Y_0:=Y_1\cup Y_2$. Proposition~\ref{prop:continuity}, together with Lemma~\ref{lem:isometry2}, imply that $u_{\infty,i}|_{\mathbb{R}^N\smallsetminus Y_0}\in\cC^{m-1}(\rn\smallsetminus Y_0)$. Consequently, 
\begin{align*}
\Theta_i:=\{x\in\R^N\smallsetminus Y_0:u_{\infty,i}(x)\neq 0\} 
\end{align*}
is a $\Gamma$-invariant nonempty open subset of $\mathbb{R}^N$ and, as $u_{\infty,i}u_{\infty,j}=0$, we have that $\Theta_i\cap\Theta_j=\emptyset$ if $i\neq j$.  We set $\Omega_i:=\operatorname{int}(\overline{\Theta_i})$. Then, every $\Omega_i$ is a nonempty  $\Gamma$-invariant open smooth subset of $\rn$, $\Omega_i\cap\Omega_j=\emptyset$ if $i\neq j$, and $u_{\infty,i}(x)=0$ in $\rn\smallsetminus\Omega_i$. By Lemma \ref{A}, $u_{\infty,i}\in D_0^{1,2}(\Omega_i)^\Gamma$ and, by \eqref{eq:limit}, $u_{\infty,i}\in\cM_{\Omega_i}$ and
\begin{equation*}
\sum_{i=1}^\ell c_{\Omega_i}^\Gamma\leq\frac{m}{N}\sum_{i=1}^\ell\|u_{\infty,i}\|^2 = c_0^\Gamma \leq \inf_{(\Phi_1,\ldots,\Phi_\ell)\in\mathcal{P}_\ell^\Gamma}\;\sum_{i=1}^\ell c_{\Phi_i}^\Gamma.
\end{equation*}
This shows that $\{\o_1,\ldots,\o_\ell\}$ is a $(\Gamma,\ell)$-optimal partition for the system \eqref{eq:system} and that $u_{\infty,i}$ is a least energy $\Gamma$-invariant solution to \eqref{eq:dirichlet} in $\o_i$. Thus, by \cite{l}, $u_{\infty,i}\in\cC^{2m,\alpha}(\overline\o_i)$ for $\alpha\in(0,1)$. This concludes the proof of statements $(a)$ and $(b)$. Statement $(c)$ follows from Lemma \ref{lem:tori}.

\end{proof}

\appendix
 
 \section{Auxiliary results}
 
\begin{lemma}\label{A}
Let $\Omega$ be an open subset of $\rn$ of class $\cC^0$. Then 
\begin{align*}
 \widetilde D_0^{m,2}(\Omega):=\{v\in D^{m,2}(\R^N)\::\:v=0\text{ in }\R^N\smallsetminus\Omega\}=D_0^{m,2}(\Omega).
 \end{align*}
\end{lemma}
\begin{proof} That $D_0^{m,2}(\Omega)\subset \widetilde D_0^{m,2}(\Omega)$ is clear. 

Consider the space $\widetilde H_0^{m}(\Omega):=\{u\in H^m(\R^N)\::\: u=0 \text{ in } \R^N\smallsetminus\Omega\}$ endowed with the standard Sobolev norm. If $\Omega$ is an open subset of $\rn$ of class $\cC^0$, then $\widetilde H^{m}_0(\Omega)=H^m_0(\Omega)$, see \cite[Thm 1.4.2.2]{g}. 

If $\Omega$ is bounded, the norm induced by the scalar product \eqref{eq:scalar_product} is equivalent to the standard Sobolev norm of $H^m_0(\Omega)$. Therefore, $\widetilde D^{m,2}_0(\Omega)=\widetilde H^{m}_0(\Omega)=H^m_0(\Omega)=D^{m,2}_0(\Omega)$.  

If $\Omega$ is unbounded we take $\varphi\in \cC_c^\infty(\R^N)$ be such that $\varphi=1$ in $B_1(0)$ and $\varphi=0$ in $\R^N\smallsetminus B_2(0)$ and set $\varphi_n(x):=\varphi(\frac{x}{n})$. If $u\in \widetilde D^{m,2}_0(\Omega)$, then $\varphi_n u$ vanishes in the complement of $\Omega_n:=\Omega\cap B_{2n}(0)$ which is of class $\cC^0$. So, by the previous case, $\varphi_n u\in D^{m,2}_0(\Omega_n)\subset D^{m,2}_0(\Omega)$ for all $n\in\N$.  It is not hard to see that $(\varphi_n u)$ is bounded in $D^{m,2}_0(\Omega)$.  But then $\varphi_n u\rightharpoonup u$ weakly in $D^{m,2}_0(\Omega)$ and, since this space is weakly closed, we conclude that $u\in D^{m,2}_0(\Omega)$. 
\end{proof}

\begin{lemma}\label{bb:lemma}
For every $\varepsilon>0$ and $i\in\N$ there are $\mu=\mu(\varepsilon,i)>1$ and $\eta=\eta(\eps,i)>0$  such that, for every $w\in C^\infty(0,\pi)$,
\begin{align}\label{i1lemma}
\frac{1}{4}|\sL^{i/2} w|^2 h&\geq \eta\left(|w^{(i)}|^2 - \mu\sum_{j=1}^{i - 1}|w^{(j)}|^2\right)\quad \text{ in }(\varepsilon,\pi-\varepsilon)\text{ for $i$ even},
\end{align}
and 
\begin{align}\label{i2lemma}
|(\sL^{(i-1)/2} w)'|^2 h &\geq \eta
\left(|w^{(i)}|^2 -\mu\sum_{j=1}^{i-1}|w^{(j)}|^2\right)\quad \text{ in }(\varepsilon,\pi-\varepsilon)\text{ for $i$ odd}.
\end{align}
\end{lemma}
\begin{proof}
Let $i\in\N$ be even and recall that $\sL^{i/2}w =\left(4\frac{d^2}{\;dt^2} + \phi(t)\frac{d}{\;dt}\right)^{i/2}w$, where $\phi\in C^{\infty}(0,\pi)$ is given by \eqref{H}. Then, $\sL^{i/2} w =\left(4\frac{d^2}{\;dt^2} + \phi(t)\frac{d}{\;dt}\right)^{i/2}w =4^{i/2} w^{(i)}+R_i,$ where, by the binomial theorem,
\begin{align*}
R_i:=\sum_{k=0}^{{i/2}-1}\binom{{i/2}}{k} 4^k\left(\phi(t)\frac{d}{\;dt}\right)^{{i/2}-k}w^{(2k)}.
\end{align*}
Fix $\varepsilon>0$. Since $\phi\in \cC^{\infty}(0,\pi)$, there is $\mu_1=\mu_1(\varepsilon,i)>0$  such that $|R_i|\leq \mu_1\sum_{k=1}^{i - 1}|w^{(k)}|$ in $(\varepsilon,\pi-\varepsilon).$  Using that $ab\leq \frac{1}{2}(a^2+b^2)$ for $a,b\in\R$, we obtain that
\begin{align*}
|\sL^{i/2} w|^2
&=4^{i}|w^{(i)}|^2+2^{i+1} w^{(i)}R_i+|R_i|^2\geq 4^{i} w^{(i)}-2^{i}(|w^{(i)}|^2+|R_i|^2)=(4^{i}-2^{i}) |w^{(i)}|^2-2^{i}|R_i|^2\\
&\geq(4^{i}-2^{i}) |w^{(i)}|^2-2^{i}\mu_1^2\left(\sum_{k=1}^{i - 1}|w^{(k)}|\right)^2
\geq (4^{i}-2^{i}) |w^{(i)}|^2 -\mu_2\sum_{k=1}^{i - 1}|w^{(k)}|^2
\end{align*}
in $(\varepsilon,\pi-\varepsilon)$ for some $\mu_2=\mu_2(\eps,i)>0$.  As a consequence, since $h\in C^\infty(0,\pi)$ is positive (see \eqref{h}),
\begin{align*}
\frac{1}{4}|\sL^{i/2} w|^2 h
&\geq (4^{i-1}-2^{i-2})\min_{[\eps,\pi-\eps]}h\left(|w^{(i)}|^2 - \mu\sum_{j=1}^{i - 1}|w^{(j)}|^2\right)\quad \text{ in }(\varepsilon,\pi-\varepsilon),
\end{align*}
for some $\mu=\mu(\eps,i)>1$ and \eqref{i1lemma} follows. Inequality \eqref{i2lemma} for $i$ odd follows similarly. 
\end{proof}

\end{document}